\documentclass{amsart}%
\usepackage{amsfonts}
\usepackage{amsmath}
\usepackage{amssymb}
\usepackage{graphicx}%
\providecommand{\U}[1]{\protect\rule{.1in}{.1in}}
\newtheorem{theorem}{Theorem}
\theoremstyle{plain}

\newtheorem{corollary}{Corollary}

\newtheorem{definition}{Definition}
\newtheorem{example}{Example}

\newtheorem{lemma}{Lemma}

\newtheorem{proposition}{Proposition}

\numberwithin{equation}{section}
\begin{document}
\title[Weakly $J$-submodules of modules over commutative rings]{Weakly $J$-submodules of modules over commutative rings}
\author{Hani A. Khashan }
\address{Department of Mathematics, Faculty of Science, Al al-Bayt University, Al
Mafraq, Jordan.}
\email{hakhashan@aabu.edu.jo.}
\author{Ece Yetkin Celikel}
\address{Department of Electrical-Electronics Engineering, Faculty of Engineering,
Hasan Kalyoncu University, Gaziantep, Turkey.}
\email{ece.celikel@hku.edu.tr, yetkinece@gmail.com.}
\thanks{This paper is in final form and no version of it will be submitted for
publication elsewhere.}
\date{March, 2021}

\begin{abstract}
Let $R$ be a commutative ring with identity and $M$ be a unitary $R$-module.
By $J(R),$ we denote the Jacobson radical of $R$. The purpose of this paper is
to introduce the concept of weakly $J$-submodules generalizing $J$-submodules.
We call a proper submodule $N$ of $M$ a weakly $J$-submodule if whenever
$0\neq rm\in N$ for $r\in R$ and $m\in M$, then $r\in(J(R)M:M)$ or $m\in N.$
Various properties and characterizations of weakly $J$-submodules are
investigated especially in the case of multiplication modules.

\end{abstract}
\maketitle

\section{Introduction}

Throughout, rings will be commutative with identity and modules are unital. We
denote by $R$ a ring, by $M$ a unitary $R$-module, and by $J(R)$ the Jacobson
radical of $R$ which is the intersection of all maximal ideals of $R$. The
radical of an ideal $I$ of $R$, denoted by $\sqrt{I}$, is defined to be the
set of all $a\in R$ for which $a^{n}\in I$ for some positive integer $n$. For
a submodule $N$ of $M$, we mean by $(N:M)$, the set of all elements $r$ of $R$
for which $rM\subseteq N$.

The concept of prime submodules, which is an important subject of module
theory, has been widely studied by various authors. Recall that a proper
submodule $N$ of an $R$-module $M$ is a (resp. weakly) prime submodule if for
$r\in R$ and $m\in M$ whenever (resp. $0\neq rm$ $\in N$) $rm$ $\in N$, then
$r\in(N:M)$ or $m\in N$, \cite{At}. In 2005, Atani and Farzalipour \cite{At2}
defined weakly primary submodules as follows: A proper submodule $N$ of $M$ is
said to be weakly primary if whenever $0\neq rm\in N$ for $r\in R$ and $m\in
M$, then $r\in\sqrt{(N:M)}$ or $m\in N$. In 2017, "Tekir et.al." \cite{Tekir}
defined the concept of $n$-ideals and $n$-submodules. The authors call a
proper submodule $N$ of $M$ an $n$-submodule if for $r\in R$ and $m\in M$,
$rm\in N$ with $r\notin\sqrt{(0:M)}$, then $m\in N$. Generalizing these
concepts, Khashan and Bani-Ata (2021) recently introduced and studied the
concepts of $J$-ideals and $J$-submodules in \cite{Hani}. According to
\cite{Hani}, a proper ideal $I$ of $R$ is said to be a $J$-ideal if whenever
$ab\in I$ for $a\in R\backslash J(R)$ and $b\in R$, then $b\in I$. As
extending this structure to the context of $J$-submodules, a proper submodule
$N$ of $M$ is called a $J$-submodule if whenever $r\in R$ and $m\in M$ such
that $rm\in N$ and $r\notin(J(R)M:M)$, then $m\in N$. As a very recent paper,
in \cite{Haniece3}, the notion of weakly $J$-ideals is presented. A weakly
$J$-ideal of $R$ is a proper ideal with the property that $a,b\in R$, $0\neq
ab\in I$ and $a\notin J(R)$ imply $b\in I$.

With $J$-submodules and weakly $J$-ideals in mind, we define the notion of
weakly $J$-submodules. We call a proper submodule $N$ of $M$ weakly
$J$-submodule if whenever $r\in R$ and $m\in M$ such that $0\neq rm\in N$ and
$r\notin(J(R)M:M)$, then $m\in N$. We obtain (Proposition \ref{Jp}) that
weakly $J$-submodules and $J$-submodules coincide in $J$-presimplifiable
modules. But, we start with some examples to show that in general, the class
of weakly $J$-submodules is a proper generalization of that of $J$-submodules.
Among many other results in this paper, in Lemma \ref{l1}, we give a condition
on a weakly $J$-submodule $N$ of an $R$-module $M$ contained in $J(R)M$ to be
a $J$-submodule. In Theorem \ref{eq1}, we conclude some equivalent statements
characterizing weakly $J$-submodules of modules. Also, we obtain some other
characterizations for weakly $J$-submodules of finitely generated faithful
multiplication modules (Corollary \ref{(N:M)}, Theorem \ref{fm} and Corollary
\ref{cm}). Moreover, we investigate the behavior of this structure under
module homomorphisms and localizations, (Propositions \ref{f}, \ref{S} and
Corollary \ref{quotient}).In Proposition \ref{wp}, we clarify a condition on
weakly primary submodules to be weakly $J$-submodules. It is shown in Theorem
\ref{max} that a maximal weakly $J$-submodule $N$ satisfying $(0:_{M}%
S)\subseteq N$ for any $S\nsubseteq J(R)$ of a finitely generated faithful
multiplication $R$-module $M$ is a $J$-submodule. Finally, in Proposition
\ref{id}, we give a characterization of weakly $J$-ideals in the idealization
ring $R(+)M$.

Let us recall that an $R$-module $M$ is called a multiplication module
provided that for every submodule $N$ of $M$, there exists an ideal
(presentation ideal) $I$ of $R$ such that $N=IM$, \cite{Bast}. If $N$ and $K$
are two submodules of a multiplication module $M$, define the multiplication
of $N$ and $K$ as $NK=IJM$ where $I$ and $J$ are the presentation ideals of
$N$ and $K$, respectively. Also, recall that the idealization ring of an
$R$-module $M$ is the set $R(+)M=R\oplus M=\left\{  (r,m):r\in R\text{, }%
m_{2}\in M\right\}  $ with coordinate-wise addition and multiplication defined
as $(r_{1},m_{1})(r_{2}m_{2})=(r_{1}r_{2},r_{1}m_{2}+r_{2}m_{1})$. If $I$ is
an ideal of $R$ and $N$ a submodule of $M$, then $I(+)N$ is an ideal of
$R(+)M$ if and only if $IM\subseteq N$. It is well known that if $I(+)N$ is an
ideal of $R(+)M$, then $\sqrt{I(+)N}=\sqrt{I}(+)M$. Moreover, we have
$J(R(+)M)=J(R)(+)M$, \cite{Anderson4}. For undefined notations or
terminologies in commutative ring and module theory, we refer the reader to
\cite{Bernard} and \cite{Sharp}.

\section{Properties of Weakly $J$-submodules}

\begin{definition}
Let $R$ be a ring and let $M$ be an $R$-module. A proper submodule $N$ of $M$
is called a weakly $J$-submodule if whenever $r\in R$ and $m\in M$ such that
$0\neq rm\in N$ and $r\notin(J(R)M:M)$, then $m\in N.$
\end{definition}

While clearly every $J$-submodule is a weakly $J$-submodule, the converse is
not necessarily true. In fact, by the definition, the zero submodule of any
module is a weakly $J$-submodule. However, $N=\left\langle 0\right\rangle $ is
not a $J$-submodule of the $%
\mathbb{Z}
$-module $%
\mathbb{Z}
_{6}$ since $2.\bar{3}\in N$ but $2\notin(J(%
\mathbb{Z}
)%
\mathbb{Z}
_{6}:%
\mathbb{Z}
_{6})=\left\langle 6\right\rangle $ and $\bar{3}\notin N$. The following is an
example of a nonzero weakly $J$-submodule that is not a $J$-submodule.

\begin{example}
\label{Ex1}Consider the $%
\mathbb{Z}
$-module $M=%
\mathbb{Z}
(+)\left(
\mathbb{Z}
_{2}\times%
\mathbb{Z}
_{2}\right)  $ and let $N=0(+)\left\langle (\bar{1},\bar{0})\right\rangle $.
We prove that $N$ is a weakly $J$-submodule of $M$. Let $r\in%
\mathbb{Z}
$ and $(r_{1},(\bar{a},\bar{b}))\in M$ such that $(0,(\bar{0},\bar{0}))\neq
r.(r_{1},(\bar{a},\bar{b}))\in N$ and $r\notin(J(%
\mathbb{Z}
)M:M)$. Then $(rr_{1},r.(\bar{a},\bar{b}))\in N\backslash(0,(\bar{0},\bar
{0}))$ and $rM\neq(0,(\bar{0},\bar{0}))$. Thus, $rr_{1}=0$, $r\neq0$ which
implies that $r_{1}=0$ and $r.(\bar{a},\bar{b})\in\left\langle (\bar{1}%
,\bar{0})\right\rangle \backslash(\bar{0},\bar{0})$. If $(\bar{a},\bar
{b})=(\bar{1},\bar{1})$ or $(\bar{0},\bar{1})$, then $r.(\bar{a},\bar{b}%
)\in\left\langle (1,0)\right\rangle $ only if $r\in\left\langle 2\right\rangle
$ and so $r.(\bar{a},\bar{b})=(\bar{0},\bar{0})$, a contradiction. Therefore,
$(\bar{a},\bar{b})\in\left\langle (\bar{1},\bar{0})\right\rangle $ and since
also $r_{1}=0$, we have $(r_{1},(\bar{a},\bar{b}))\in N$ as needed. On the
other hand, $N$ is not a $J$-submodule of $M$. For example, $2.(0,(\bar
{1},\bar{1}))=(0,(\bar{0},\bar{0}))\in N$ with $2\notin(J(%
\mathbb{Z}
)M:M)$ but $(0,(\bar{1},\bar{1}))\notin N$.
\end{example}

We next give a rather general condition on a weakly $J$-submodule $N$ of an
$R$-module $M$ contained in $J(R)M$ to be a $J$-submodule.

\begin{lemma}
\label{l1}Let $N$ be a weakly $J$-submodule of an $R$-module $M$ such that
$N\subseteq J(R)M$. If $N$ is not a $J$-submodule, then $(N:M)N=0$. In
particular, $(N:M)^{2}\subseteq Ann(M).$ Hence a weakly $J$-submodule
$N\subseteq J(R)M$ with $(N:M)N\neq0$ is $J$-submodule.
\end{lemma}

\begin{proof}
Assume on the contrary $(N:M)N\neq0$. We show that $N$ is a $J$-submodule. Let
$r\in R$ and $m\in M$ with $rm\in N$. If $rm\neq0$, then clearly we are done.
So, assume that $rm=0$. We have the following cases:

\textbf{Case 1.} Let $rN\neq0.$ Then $rn\neq0$ for some $n\in N.$ Hence,
$0\neq rn=r(n+m)\in N$ which implies either $r\in(J(R)M:M)$ or $n+m\in N$.
Thus, $r\in(J(R)M:M)$ or $m\in N$.

\textbf{Case 2. }Let $(N:M)m\neq0.$ Then there exists $a\in(N:M)$ satisfying
$am\neq0.$ Hence, $0\neq am=(r+a)m\in N$ which follows either $r+a\in
(J(R)M:M)$ or $m\in N.$ By assumption, $a\in(N:M)\subseteq(J(R)M:M)$ and so
$r\in(J(R)M:M)$ or $m\in N$.

\textbf{Case 3.} Let $rN=0$ and $(N:M)m=0$. Since $(N:M)N\neq0$, then
$bn_{1}\neq0$ for some $b\in(N:M)\subseteq(J(R)M:M)$ and $n_{1}\in N$. So,
$0\neq bn_{1}=(r+b)(n_{1}+m)\in N$ which implies $r+b\in(J(R)M:M)$ or
$n_{1}+m\in N.$ Thus, $r\in(J(R)M:M)$ or $m\in N$ and we are done.

In particular, $(N:M)^{2}\subseteq((N:M)N:M)=(0:M)=Ann(M).$
\end{proof}

In the following result, we give several characterizations for weakly
$J$-submodules of an $R$-module $M$.

\begin{theorem}
\label{eq1} For a proper submodule $N$ of an $R$-module $M$, the following
statements are equivalent:
\end{theorem}

\begin{enumerate}
\item $N$ is a weakly $J$-submodule of $M$.

\item For $m\in M\backslash N,$ $(N:Rm)\subseteq(J(R)M:M)\cup(0:Rm)$.

\item Whenever $I$ is an ideal of $R$ and $m\in M$ with $0\neq Im\subseteq N$,
then $I\subseteq(J(R)M:M)$ or $m\in N$.

\item Whenever $I$ is an ideal of $R$ and $K$ is a submodule of $M$ with
$0\neq IK\subseteq N$, then $I\subseteq(J(R)M:M)$ or $K\subseteq N$.
\end{enumerate}

\begin{proof}
(1)$\Rightarrow$(2) Suppose $N$ is a weakly $J$-submodule of $M$ and let $m\in
M\backslash N$. Let $r\in(N:Rm)$ so that $rm\in N$. If $rm=0$, then
$r\in(0:Rm)$. If $rm\neq0$, then $r\in(J(R)M:M)$ since $N$ is a weakly
$J$-submodule and $m\notin N$. Thus, $(N:Rm)\subseteq(J(R)M:M)\cup(0:Rm)$ as needed.

(2)$\Rightarrow$(3) Suppose $I$ is an ideal of $R$ and $m\in M$ with $0\neq
Im\subseteq N$ and $m\notin N$. Then $I\subseteq(N:Rm)$ with $I\nsubseteq
(0:Rm)$. By assumption, we conclude that $I\subseteq(J(R)M:M)$ and we are done.

(3)$\Rightarrow$(4) Assume on the contrary that there are an ideal $I$ of $R$
and a submodule $K$ of $M$ such that $0\neq IK\subseteq N$ but $I\nsubseteq
(J(R)M:M)$ and $K\nsubseteq N$. Then there exists $m\in K$ such that $0\neq
Im\subseteq N$. Since $I\nsubseteq(J(R)M:M)$, then by (3), we get $m\in N$.
Now, choose an element $m^{\prime}\in K\backslash N$ and note that
$Im^{\prime}=0$ since otherwise, if $Im^{\prime}\neq0$, then $m^{\prime}\in
N$, a contradiction. Thus, $0\neq Im=I(m^{\prime}+m)\subseteq N$ and as
$I\nsubseteq(J(R)M:M)$, we have $m^{\prime}+m\in N$. Therefore, again we get
$m^{\prime}\in N$, a contradiction.

(4)$\Rightarrow$(1) Suppose $0\neq rm\in N$ for $r\in R$ and $m\in M$. By
putting $I=Rr$ and $K=Rm$ in (4), the claim is clear.
\end{proof}

Let $M$ be a finitely generated faithful multiplication $R$-module, $N$ be a
proper submodule of $M$ and $I$ be an ideal of $R$. It is known from
\cite{Smith} that $(IN:M)=I(N:M)$ and in particular, $(IM:M)=I$.

\begin{proposition}
\label{IM}Let $M$ be a finitely generated faithful multiplication $R$-module
and $I$ be an ideal of $R$. Then $I$ is a weakly $J$-ideal of $R$ if and only
if $IM$ is a weakly $J$-submodule of $M$.
\end{proposition}

\begin{proof}
Let $I$ be a weakly $J$-ideal of $R$. If $IM=M$, then $I=(IM:M)=R$, a
contradiction. Thus, $IM$ is proper in $M$. Now, let $r\in R$ and $m\in M$
such that $0\neq rm\in IM$ and $r\notin(J(R)M:M)=J(R)$. If $((rm):M)=0$, then
$rm=((rm):M)M=0$, a contradiction. Thus, $0\neq r((m):M)=((rm):M)\subseteq
(IM:M)=I$. As $I$ is a weakly $J$-ideal of $R$, we conclude that
$((m):M)\subseteq I$ and so $m\in((m):M)M\subseteq IM$. Conversely, suppose
$IM$ is a weakly $J$-submodule of $M$ and let $a,b\in R$ such that $0\neq
ab\in I$ with $a\notin J(R)$. Then $0\neq abM\subseteq IM$ (if $abM=0$, then
$\left\langle ab\right\rangle =(\left\langle ab\right\rangle M:M)=0$, a
contradiction). Since also $a\notin(J(R)M:M)$, then $bM\subseteq IM$. It
follows that $b\in(bM:M)\subseteq(IM:M)=I$ as needed.
\end{proof}

As $N=(N:M)M$ for any submodule $N$ of a multiplication $R$-module $M$, we
have the following corollary of Proposition \ref{IM}.

\begin{corollary}
\label{(N:M)}Let $M$ be a finitely generated faithful multiplication
$R$-module and $N$ be a submodule of $M$. The following are equivalent:

\begin{enumerate}
\item $N$ is a weakly $J$-submodule of $M$.

\item $(N:M)$ is weakly $J$-ideal of $R$.

\item $N=IM$ for some weakly $J$-ideal $I$ of $R$.
\end{enumerate}
\end{corollary}

\begin{lemma}
\label{l2}\cite{Haniece3} If $I$ is a weakly $J$-ideal of a ring $R$, then
$I\subseteq J(R).$
\end{lemma}

If $M$ is a multiplication $R$-module and $N=IM$, $K=JM$ are two submodules of
$M$, then the product $NK$ of $N$ and $K$ is defined as $NK=(IM)(JM)=(IJ)M$.
In particular, if $m_{1},m_{2}\in M$, then $m_{1}m_{2}=\left\langle
m_{1}\right\rangle \left\langle m_{2}\right\rangle $. Recall from \cite{Lee}
that an $R$-module $M$ is said to be reduced if whenever $a\in R$, $m\in M$
with $a^{2}m=0$ implies $am=0.$

Next, we show that the condition $N\subseteq J(R)M$ in Lemma \ref{l1} can be
omitted if $M$ is a finitely generated faithful multiplication $R$-module.

\begin{corollary}
\label{reduced}Let $N$ be a weakly $J$-submodule of a finitely generated
faithful multiplication $R$-module $M$ that is not a $J$-submodule. Then
$(N:M)N=0$. In particular, we have $N^{2}=0$. Additionally, if $M$ is reduced,
then $N=0$.
\end{corollary}

\begin{proof}
From Corollary \ref{(N:M)}, Lemma \ref{l1} and Lemma \ref{l2}, we conclude
$(N:M)N=0$. Moreover, $N^{2}=(N:M)^{2}M=(N:M)(N:M)M=(N:M)N=0$. Now, suppose
that $M$ is reduced and let $a\in(N:M)$. Then for all $m\in\dot{M}$,
$a^{2}m\in(N:M)N=0$ which implies $am=0$. Thus $N=(N:M)M=0$ which completes
the proof.
\end{proof}

It follows by Corollary \ref{reduced} that every non-zero weakly $J$-submodule
of a finitely generated faithful reduced multiplication module is a $J$-submodule.

We recall that a submodule $N$ of an $R$-module $M$ is called pure if
$IN=N\cap IM$ for every ideal $I$ of $R$.

\begin{proposition}
Let $M$ be a finitely generated faithful multiplication $R$-module and $N$ be
a pure submodule of $M$. If $I$ is a weakly $J$-ideal of $R$ and $N$ is a
weakly $J$-submodule of $M$, then $IN$ is a weakly $J$-submodule of $M$.
\end{proposition}

\begin{proof}
Let $r\in R$ and $m\in M$ such that $0\neq rm\in IN$ and $r\notin
(J(R)M:M)=J(R)$. Then clearly $0\neq r((m):M)=((rm):M)\in
(IN:M)=I(N:M)\subseteq I\cap(N:M)$. Since $(N:M)$ is a weakly $J$-ideal by
Corollary \ref{(N:M)}, and $I$ is also a weakly $J$-ideal, then
$((m):M)\subseteq I\cap(N:M)$. As $((m):M)\subseteq(N:M)$, we conclude that
$m\in((m):M)M\subseteq(N:M)M=N$. Moreover, $((m):M)\subseteq I$ implies that
$m\in IM$. Thus, $m\in IM\cap N=IN$ since $N$ is pure in $M$.
\end{proof}

\begin{lemma}
\cite{Majed}\label{Majed} Let $N$ is a submodule of a faithful multiplication
$R$-module $M$. If $I$ is a finitely generated faithful multiplication ideal
of $R$, then

\begin{enumerate}
\item $N=(IN:_{M}I)$.

\item If $N\subseteq IM$, then $(JN:_{M}I)=J(N:_{M}I)$ for any ideal $J$ of
$R$.
\end{enumerate}
\end{lemma}

\begin{proposition}
Let $I$ be a finitely generated faithful multiplication ideal of a ring $R$
and $N$ be a submodule of a faithful multiplication $R$-module $M$. If $IN$ is
a weakly $J$-submodule of $M$, then either $I$ is a weakly $J$-ideal of $R$ or
$N$ is a weakly $J$-submodule of $M$.
\end{proposition}

\begin{proof}
If $N=M$, then $I=I(N:M)=(IN:M)$ is a weakly $J$-ideal of $R$ by Corollary
\ref{(N:M)}. Suppose $N\varsubsetneq M$. By Lemma \ref{Majed}, we have
$N=(IN:_{M}I)$ and so it can be easily verified that $(N:M)=((IN:_{M}%
I):M)=(I(N:M):I)$. Let $a,b\in R$ such that $0\neq ab\in(N:M)$ and $a\notin
J(R)$. Since $I$ is faithful, then $0\neq Iab\subseteq I(N:M)=(IN:M)$. Thus,
$Ib\subseteq I(N:M)$ as $(IN:M)$ is a weakly $J$-ideal. It follows again by
Lemma \ref{Majed} that $b\in(I(N:M):I)=(N:M)$ and so $(N:M)$ is a weakly
$J$-ideal of $R$. Therefore, $N$ is a weakly $J$-submodule of $M$ by Corollary
\ref{(N:M)}.
\end{proof}

Let $N$ be a submodule of an $R$-module $M$ and $I$ be an ideal of $R$. The
residual of $N$ by $I$ is the set $(N:_{M}I)=\{m\in M:Im\subseteq N\}$. It is
clear that $(N:_{M}I)$ is a submodule of $M$ containing $N$. More generally,
for any subset $S\subseteq R$, $(N:_{M}S)$ is a submodule of $M$ containing
$N$.

\begin{proposition}
Let $M$ be a finitely generated faithful multiplication $R$-module and $I$ be
an ideal of $R$. Then

\begin{enumerate}
\item If $N$ is a weakly $J$-submodule of $M$ and $(N:_{M}I)\neq M$, then
$(N:_{M}I)$ is also a weakly $J$-submodule of $M$.

\item If $I$ is finitely generated faithful multiplication and $(N:_{M}I)\neq
M$, then $N$ is a weakly $J$-submodule of $IM$ if and only if $(N:_{M}I)$ is a
weakly $J$-submodule of $M$.
\end{enumerate}
\end{proposition}

\begin{proof}
(1) Suppose $N$ is a weakly $J$-submodule of $M$ and $(N:_{M}I)$ is proper in
$M$. Let $r\in R$ and $m\in M$ such that $0\neq rm\in$ $(N:_{M}I)$ and
$r\notin(J(R)M:M)$. Then $0\neq rmI\subseteq N$. Indeed, if $rmI=0$, then
$0=rmI\subseteq I(\left\langle rm\right\rangle :M)M$. Since $M$ is faithful,
then $rI(\left\langle m\right\rangle :M)=I(\left\langle rm\right\rangle
:M)=0$. As $I$ is also faithful, it follows that $rm\in(\left\langle
rm\right\rangle :M)M\subseteq(0:I)M=0$, a contradiction. Since $N$ is a weakly
$J$-submodule, we get $mI\subseteq N$ and so $m\in(N:_{M}I)$ as needed.

(2) Suppose $N$ is a weakly $J$-submodule of $IM$. Let $0\neq rm\in$
$(N:_{M}I)$ for $r\in R$ and $m\in M$ such that $r\notin(J(R)M:M)$. Then
similar to the proof of (1), we have $0\neq rmI\subseteq N$. Suppose on the
contrary that $r\in(J(R)IM:IM)$. Then by (1) of Lemma \ref{Majed},
$rM=r(IM:_{M}I)\subseteq(rIM:_{M}I)\subseteq(J(R)IM:_{M}I)=J(R)M$, a
contradiction. Thus, $r\notin(J(R)IM:IM)$ and so by assumption $mI\subseteq
N$. Therefore, $m\in(N:_{M}I)$ and $(N:_{M}I)$ is a weakly $J$-submodule.
Conversely, suppose $(N:_{M}I)$ is a weakly $J$-submodule of $M$. Then $N$ is
proper in $IM$ since otherwise Lemma \ref{Majed} implies $(N:_{M}I)=(IM:I)=M$,
a contradiction. Let $r\in R$ and $am\in IM$ ($a\in I$) such that $0\neq
ram\in N$ and $r\notin(J(R)IM:IM)$. Then $r(\left\langle am\right\rangle
:_{M}I)\subseteq(\left\langle ram\right\rangle :_{M}I)\subseteq(N:_{M}I)$.
Moreover, $r(\left\langle am\right\rangle :_{M}I)\neq0$ since otherwise we
have again by (1) of Lemma \ref{Majed}, $ram\in r(Iam:_{M}I)\subseteq
r(\left\langle am\right\rangle :_{M}I)=0$, a contradiction. Since also clearly
$r\notin(J(R)M:M)$, then by Theorem \ref{eq1}, $(\left\langle am\right\rangle
:_{M}I)\subseteq(N:_{M}I)$. Thus, $am\in(\left\langle am\right\rangle
I:_{M}I)=I((\left\langle am\right\rangle :_{M}I))\subseteq I(N:_{M}%
I)=(IN:_{M}I)=N$ by Lemma \ref{Majed} and we are done.
\end{proof}

\begin{proposition}
\label{(N:S)}Let $M$ be a finitely generated faithful multiplication
$R$-module and $S$ be a subset of $R$ with $S\nsubseteq J(R)$. If $N$ is a
weakly $J$-submodule of $M$ and $(0:_{M}S)\subseteq N$, then $(N:_{M}S)$ is a
weakly $J$-submodule of $M$
\end{proposition}

\begin{proof}
Suppose $N$ is weakly $J$-submodule and $(0:_{M}S)\subseteq N$. If
$(N:_{M}S)=M$ and $m\in M$, then $Sm\subseteq N$. If $Sm=0$, then $m\in
(0:_{M}S)\subseteq N$. If $Sm\neq0$, then $S\nsubseteq J(R)=(J(R)M:M)$ implies
also that $m\in N$. Thus, $M=N$ which is a contradiction. Now, let $r\in R$
and $m\in M$ such that $0\neq rm\in$ $(N:_{M}S)$ and $r\notin(J(R)M:M)$. If
$rSm=0$, then $0\neq rm\in N$ and so $m\in N\subseteq(N:_{M}S)$. Otherwise, we
have $0\neq rSm\subseteq N$ and so $Sm\subseteq N$. Hence, again $m\in
(N:_{M}S)$ and the result follows.
\end{proof}

We call a proper submodule $N$ of an $R$-module $M$ a maximal weakly
$J$-submodule if there is no weakly $J$-submodule which contains $N$ properly.

\begin{theorem}
\label{max}Let $M$ be a finitely generated faithful multiplication $R$-module
and $N$ be a submodule of $M$ satisfying $(0:_{M}S)\subseteq N$ for any
$S\nsubseteq J(R)$. If $N$ is a maximal weakly $J$-submodule of $M$, then $N$
is a $J$-submodule of $M$.
\end{theorem}

\begin{proof}
Suppose $N$ is a maximal weakly $J$-submodule of $M$ satisfying $(0:_{M}%
S)\subseteq N$ for any $S\nsubseteq J(R)$. Let $rm\in N$ for $r\in R$ and
$m\in M$ such that $r\notin J(R)=(J(R)M:M)$. Then $(N:_{M}r)$ is also a weakly
$J$-submodule of $M$ by Proposition \ref{(N:S)}. By the maximality of $N$, we
get $m\in(N:_{M}r)=N$ and so $N$ is a $J$-submodule of $M$.
\end{proof}

Let $M$ be an $R$-module. By $J(M)$, we denote the Jacobson radical of $M$
which is the intersection of all maximal submodules of $M$. For a finitely
generated faithful multiplication module $M$, it is well-known that
$J(M)=J(R)M$, \cite{Bast}.

The next result gives a characterization for weakly $J$-submodules of finitely
generated faithful multiplication $R$-modules.

\begin{theorem}
\label{fm}Let $M$ be a finitely generated faithful multiplication $R$-module.
Then the following are equivalent:
\end{theorem}

\begin{enumerate}
\item $N$ is a weakly $J$-submodule of $M$.

\item Whenever $K,L$ are submodules of $M$ and $0\neq KL\subseteq N$, then
$K\subseteq J(M)$ or $L\subseteq N$.
\end{enumerate}

\begin{proof}
(1)$\Rightarrow$(2) Let $N$ be a weakly $J$-submodule of $M$. Suppose that
$I_{1}$ and $I_{2}$ are the presentation ideals of $K$ and $L$, respectively.
Then $0\neq I_{1}I_{2}M\subseteq N$ which implies $I_{1}\subseteq(J(R)M:M)$ or
$I_{2}M\subseteq N$ by Theorem \ref{eq1}. Thus, $K=I_{1}M\subseteq J(R)M=J(M)$
or $L=I_{2}M\subseteq N$, as needed.

(2)$\Rightarrow$(1) Suppose that $0\neq IL\subseteq N$ for some ideal $I$ of
$R$ and submodule $L$ of $M$. Assume that $L\nsubseteq N$. By taking $K=IM$ in
(2), we conclude that $K=IM\subseteq J(M)$. Thus, we have $I\subseteq
(J(R)M:M)$ and the result follows by Theorem \ref{eq1}.
\end{proof}

\begin{corollary}
\label{cm}Let $N$ be a proper submodule of a finitely generated faithful
multiplication $R$-module $M$. Then $N$ is a weakly $J$-submodule of $M$ if
and only if whenever $m_{1},m_{2}\in M$ with $0\neq m_{1}m_{2}\in N$, then
$m_{1}\in J(M)$ or $m_{2}\in N$.
\end{corollary}

\begin{proposition}
If $N$ is a weakly $J$-submodule of a finitely generated faithful
multiplication $R$-module $M$, then $N\subseteq J(M)$. Moreover, if $K$ is a
weakly $J$-submodule of $M$ which is not a $J$-submodule, then $NK=0$.
\end{proposition}

\begin{proof}
Let $N$ be a weakly $J$-submodule of $M$. Then $(N:M)$ is a weakly $J$-ideal
of $R$ by Corollary \ref{(N:M)}. Hence, $(N:M)\subseteq J(R)$ by Lemma
\ref{l2}, and so $N=(N:M)M\subseteq J(R)M=J(M)$. Now, suppose $K$ is a weakly
$J$-submodule of $M$ which is not a $J$-submodule. Since $(N:M)\subseteq
J(R)$, the claim follows from Lemma \ref{l1}.
\end{proof}

\begin{proposition}
\label{f}Let $\varphi:M_{1}\longrightarrow M_{2}$ be an $R$-module
homomorphism. Then

\begin{enumerate}
\item If $\varphi$ is surjective and $N$ is a weakly J-submodule of $M_{1}$
with $\ker\left(  \varphi\right)  \subseteq N$, then $\varphi\left(  N\right)
$ is a weakly J-submodule of $M_{2}$.

\item If $\varphi$ is one-to-one and $K$ is a weakly J-submodule of $M_{2}$,
then $\varphi^{-1}\left(  K\right)  $ is a weakly J-submodule of $M_{1}$.
\end{enumerate}
\end{proposition}

\begin{proof}
(1) Suppose $\varphi\left(  N\right)  =M_{2}=\varphi\left(  M_{1}\right)  $
and let $m_{1}\in M_{1}$. Then $\varphi\left(  m_{1}\right)  =\varphi\left(
n\right)  $ for some $n\in N$ and so $m_{1}-n\in\ker\left(  \varphi\right)
\subseteq N$. So, $m_{1}\in N$ and $N=M_{1}$ which is a contradiction. Hence,
$\varphi\left(  N\right)  $ is proper in $M_{2}$. Let $r\in R$ and $m_{2}\in
M_{2}$ such that $0\neq rm_{2}\in\varphi\left(  N\right)  $ and $r\notin
\left(  J\left(  R\right)  M_{2}:M_{2}\right)  $. Choose $m_{1}\in M_{1}$ such
that $\varphi\left(  m_{1}\right)  =m_{2}$. Then $0\neq rm_{2}=r\varphi\left(
m_{1}\right)  =\varphi\left(  rm_{1}\right)  \in\varphi\left(  N\right)  $
which implies $0\neq rm_{1}\in N$ as $\ker\left(  \varphi\right)  \subseteq
N$. If $rM_{1}\subseteq J\left(  R\right)  M_{1}$, then $rM_{2}=r\varphi
\left(  M_{1}\right)  =\varphi(rM_{1})\subseteq\varphi\left(  J\left(
R\right)  M_{1}\right)  =J\left(  R\right)  \varphi\left(  M_{1}\right)
=J\left(  R\right)  M_{2}$ which is a contradiction. Thus, $r\notin\left(
J\left(  R\right)  M_{1}:M_{1}\right)  $. Since $N$ is a weakly $J$-submodule,
then $m_{1}\in N$ and so $m_{2}=\varphi\left(  m_{1}\right)  \in\varphi(N)$ as required.

(2) Let $r\in R$ and $m_{1}\in M_{1}$ such that $0\neq rm_{1}\in\varphi
^{-1}\left(  K\right)  $ and $r\notin\left(  J\left(  R\right)  M_{1}%
:M_{1}\right)  $. Since $K\operatorname{erf}=0$, then $0\neq r\varphi\left(
m_{1}\right)  =\varphi\left(  rm_{1}\right)  \in K$. Moreover, we have
$r\notin\left(  J\left(  R\right)  M_{2}:M_{2}\right)  $. Indeed, if
$rM_{2}\subseteq J\left(  R\right)  M_{2}$, then $r\varphi\left(
M_{1}\right)  \subseteq J\left(  R\right)  \varphi\left(  M_{1}\right)  $ and
so $\varphi\left(  rM_{1}\right)  \subseteq\varphi\left(  J\left(  R\right)
M_{1}\right)  $. Now, if $x\in rM_{1}$, then $\varphi\left(  x\right)
\in\varphi\left(  rM_{1}\right)  \subseteq\varphi\left(  J\left(  R\right)
M_{1}\right)  $ and so $x-y\in\ker\left(  \varphi\right)  \subseteq J\left(
R\right)  M_{1}$\ for some $y\in J\left(  R\right)  M_{1}$. Hence, $x\in
J\left(  R\right)  M_{1}$ and $rM_{1}\subseteq J\left(  R\right)  M_{1}$, a
contradiction. Since $K$ is a weakly $J$-submodule of $M_{2}$, then
$\varphi\left(  m_{1}\right)  \in K$, and thus $m_{1}\in\varphi^{-1}\left(
K\right)  $ and we are done.
\end{proof}

\begin{corollary}
\label{quotient}Let $N$ and $L$ be two submodules of an $R$-module $M$ with
$L\subseteq N$.
\end{corollary}

\begin{enumerate}
\item If $N$ is a weakly $J$-submodule of $M$, then $N/L$ is a weakly
$J$-submodule of $M/L$.

\item If $L$ is a weakly $J$-submodule of $M$ and $N/L$ is a weakly
$J$-submodule of $M/L$, then $N$ is a weakly $J$-submodule of $M$.

\item If $L$ is a $J$-submodule of $M$ and $N/L$ is a weakly $J$-submodule of
$M/L$, then $N$ is a $J$-submodule of $M$.
\end{enumerate}

\begin{proof}
(1) Clear by Proposition \ref{f}.

(2) Suppose that $0\neq rm\in N$ and $r\notin(J(R)M:M)$. If $rm\in L$, then
$m\in L\subseteq N$ since $L$ is a weakly $J$-submodule. So, assume that
$rm\notin L$. One can easily observe that $r\notin(J(R)M/N:M/N)$. Since $N/L$
is a weakly $J$-submodule of $M/L$ and $L\neq r(m+L)\in N/L$, then $m+L\in
N/L$. Therefore, $m\in N$ and $N$ is a weakly $J$-submodule of $M$ as needed.

(3) Similar to (2).
\end{proof}

\begin{proposition}
\label{int}Let $\left\{  N_{i}:i\in\Delta\right\}  $ be a non empty family of
weakly $J$-submodules of an $R$-module $M$. Then $\bigcap\limits_{i\in\Delta}$
$N_{i}$ is a weakly $J$-submodule.
\end{proposition}

\begin{proof}
Suppose that $0\neq rm\in\bigcap\limits_{i\in\Delta}$ $N_{i}$ for some $r\in
R\backslash(J(R)M:M)$, $m\in M$. Since $N_{i}$ is a weakly $J$-submodule of
$M$, for all $i\in\Delta$, we have $m\in N_{i}$. Thus, $m\in\bigcap
\limits_{i\in\Delta}$ $N_{i}$, so we are done.
\end{proof}

We recall that a submodule $N$ of an $R$-module $M$ is called small (or
superfluous) in $M$ in case whenever $N+K=M$ for a submodule $K$ of $M$, then
$K=M$.

\begin{lemma}
\label{small}Every weakly $J$-submodule of a finitely generated faithful
multiplication $R$-module is small.
\end{lemma}

\begin{proof}
Let $N$ be a weakly $J$-submodule of an $R$-module $M$ and $K$ be a submodule
of $M$ with $N+K=M$. Then clearly, $(N:M)+(K:M)=(N+K:M)=R$. Since $(N:M)$ is a
weakly $J$-ideal of $R$ by Corollary \ref{(N:M)}, then $(K:M)=R$ by
\cite[Lemma 2]{Haniece3}. Thus, $K=(K:M)M=M$ as needed.
\end{proof}

\begin{proposition}
\label{sum}Let $N_{1}$ and $N_{2}$ be weakly $J$-submodules of an $R$-module
$M$. Then $N_{1}+N_{2}$ is a weakly $J$-submodule of $M$.
\end{proposition}

\begin{proof}
Suppose that $N_{1}$ and $N_{2}$ are weakly $J$-submodules of $M$. If
$N_{1}+N_{2}=M$, then $N_{1}=N_{2}=M$ by Lemma \ref{small}, a contradiction.
Now, $N_{1}/(N_{1}\cap N_{2})$ is a weakly $J$-submodule of $M/(N_{1}\cap
N_{2})$ by Corollary \ref{quotient} (1). Hence, by isomorphism $N_{1}%
/(N_{1}\cap N_{2})\cong(N_{1}+N_{2})/N_{2}$, we get $(N_{1}+N_{2})/N_{2}$ is a
weakly $J$-submodule of of $M/N_{2}$. Again, by Corollary \ref{quotient} (2),
$N_{1}+N_{2}$ is a weakly $J$-submodule of $M$.
\end{proof}

Following \cite{Haniece2}, an $R$-module $M$ is called $J$-presimplifiable if
$Z(M)\subseteq(J(R)M:M)$. In \cite[Theorem 1]{Haniece2}, it is proved that if
$N$ is a submodule of $M$ with $N\subseteq J(R)M$, then $N$ is a $J$-submodule
if and only if $M/N$ is a non-zero $J$-presimplifiable. In particular, $M$ is
$J$-presimplifiabe if and only if $\{0\}$ is a $J$-submodule of $M$.

\begin{proposition}
\label{Jp}Every weakly $J$-submodule of a $J$-presimplifiable module is a $J$-submodule.
\end{proposition}

\begin{proof}
Let $N$ be a weakly $J$-submodule of an $J$-presimplifiable $R$-module $M$.
Since $\{0\}$ is a a $J$-submodule of $M$, the result follows by (3) of
Corollary \ref{quotient}.
\end{proof}

\begin{proposition}
\label{d}Let $M_{1},M_{2},...,M_{k}$ be $R$-modules and consider the
$R$-module $M=M_{1}\times\cdots\times M_{k}$. If $N=N_{1}\times\cdots\times
N_{k}$ is a weakly $J$-submodule of $M$, then $N_{i}$ is a weakly
$J$-submodule of $M_{i}$ for all $i$ such that $N_{i}\neq M_{i}$.
\end{proposition}

\begin{proof}
Let $N_{i}\neq M_{i}$ for some $i=1,...,k$. We show that $N_{i}$ is a weakly
$J$-submodule of $M_{i}$. Let $r\in R$ and $m_{i}\in M_{i}$ such that $0\neq
rm_{i}\in N_{i}$ and $r\notin(J(R)M_{i}:M_{i})$. Since $0\neq r.(0,...,0,m_{i}%
,0...,0)\in N$ and clearly $r\notin(J(R)M:M),$ we conclude $(0,...,0,m_{i}%
,0...,0)\in N$. Thus, $m_{i}\in N_{i}$ and $N_{i}$ is a weakly $J$-submodule
of $M_{i}$.
\end{proof}

If $N_{1}$ and $N_{2}$ are weakly $J$-submodules of $R$-modules $M_{1}$ and
$M_{2}$ respectively, then $N_{1}\times N_{2}$ need not be a weakly
$J$-submodule of $M_{1}\times M_{2}$. For example, consider the $%
\mathbb{Z}
$-modules $M_{1}=%
\mathbb{Z}
_{6}$ and $M_{2}=%
\mathbb{Z}
(+)\left(
\mathbb{Z}
_{2}\times%
\mathbb{Z}
_{2}\right)  $. Then $N_{1}=\left\langle \bar{0}\right\rangle $ and
$N_{2}=0(+)\left\langle (\bar{1},\bar{0})\right\rangle $ are weakly
$J$-submodules of $M_{1}$ and $M_{2}$ respectively. On the other hand,
$N_{1}\times N_{2}$ is not a weakly $J$-submodule. Indeed, $3\cdot(\bar
{2},(0,(1,0)))=(\bar{0},(0,(1,0)))\in N_{1}\times N_{2}$ but clearly
$3\notin(J(%
\mathbb{Z}
)(M_{1}\times M_{2}):(M_{1}\times M_{2}))$ and $(\bar{2},(0,(1,0)))\notin
N_{1}\times N_{2}$.

Let $I$ be a proper ideal of $R$ and $N$ be a submodule of an $R$-module $M$ .
In the following proposition, the notations $Z_{I}(R)$ and $Z_{N}(M)$ denote
the sets $\{r\in R:rs\in I$ for some $s\in R\backslash I\}$ and $\{r\in
R:rm\in N$ for some $m\in M\backslash N\}$.

\begin{proposition}
\label{S}Let $S$ be a multiplicatively closed subset of a ring $R$ such that
$S^{-1}(J(R))=J(S^{-1}R)$ and $M$ be an $R$-module. Then

\begin{enumerate}
\item If $N$ is a weakly $J$-submodule of $M$ and $S^{-1}N\neq S^{-1}M$, then
$S^{-1}N$ is a weakly $J$-submodule of the $S^{-1}R$-module $S^{-1}M$.

\item If $S^{-1}N$ is a weakly $J$-submodule of $S^{-1}M$ and $S\cap
Z(M)=S\cap Z_{(J(R)M:M)}(R)=S\cap Z_{N}(M)=\emptyset$, then $N$ is a weakly
$J$-submodule of $M$.
\end{enumerate}
\end{proposition}

\begin{proof}
(1) Suppose that $0\neq\frac{r}{s_{1}}\frac{m}{s_{2}}\in S^{-1}N$ for
$\frac{r}{s_{1}}\in S^{-1}R$ and $\frac{m}{s_{2}}\in S^{-1}M$ . Then $0\neq
urm\in N$ for some $u\in S$. Since $N$ is a weakly $J$-submodule, we have
either $ur\in(J(R)M:M)$ or $m\in N$. Hence, we conclude that either $\frac
{r}{s_{1}}=\frac{ur}{us_{1}}\in S^{-1}(J(R)M:M)=(S^{-1}J(R)$ $S^{-1}%
M:S^{-1}M)=(J(S^{-1}R)$ $S^{-1}M:S^{-1}M)$ or $\frac{m}{s_{2}}\in S^{-1}N$ as required.

(2) Let $r\in R$ and $m\in M$ such that $0\neq rm\in N$. From our assumption
$S\cap Z(M)=\emptyset$, we have $0\neq\frac{r}{1}\frac{m}{1}\in S^{-1}N$ which
implies that $\frac{r}{1}\in(J(S^{-1}R)S^{-1}M:S^{-1}M)=S^{-1}(J(R)M:M)$ or
$\frac{m}{1}\in S^{-1}N.$ Thus, either $ur\in(J(R)M:M)$ for some $u\in S$ or
$vm\in N$ for some $v\in S$. Since $S\cap Z_{(J(R)M:M)}(R)=S\cap
Z_{N}(M)=\emptyset$, we conclude that $r\in(J(R)M:M)$ or $m\in N$. Therefore,
$N$ is a weakly $J$-submodule of $M$.
\end{proof}

Recall that a proper submodule $N$ of an $R$-module $M$ is said to be weakly
primary if whenever , $0\neq rm\in N$ for $r\in R$ and $m\in M$, then
$r\in\sqrt{(N:M)}$ or $m\in N$, \cite{At2}. In the following, we show that
every weakly primary submodule $N$ of $M$ with $(N:M)$ contained in $J(R)$ is
a weakly $J$-submodule of $M$.

\begin{proposition}
\label{wp}Let $M$ be an $R$-module. If $N$ is a weakly primary submodule of
$M$ such that $(N:M)\subseteq J(R)$, then $N$ is a weakly $J$-submodule of $M$.
\end{proposition}

\begin{proof}
Let $r\in R$ and $m\in M$ such that $0\neq rm\in N$ and $r\notin(J(R)M:M)$.
Then $r\notin J(R)$ and so by assumption, $r\notin\sqrt{(N:M)}$. It follows
that $m\in N$ as $N$ is a weakly primary submodule of $M$.
\end{proof}

Let $M$ be an $R$-module. Next, we justify the relation between the weakly
$J$-ideals of the idealization ring $R(+)M$ and the weakly $J$-submodules of
$M$.

\begin{proposition}
\label{id}Let $I$ be an ideal of a ring $R$ and $N$ a proper submodule of an
$R$-module $M$.

\begin{enumerate}
\item If $I(+)N$ is a weakly $J$-ideal of $R(+)M$, then $I$ is a weakly
$J$-ideal of $R$ and $N$ is a weakly $J$-submodule of $M$.

\item Suppose $I$ is a $J$-ideal of $R$ and $N$ is a weakly $J$-submodule of
$M$. If for $r\in R,m\in M$ with $rm=0$ but $r\notin J(R)$ and $m\notin N$,
$(I:\left\langle r\right\rangle )=0$, then $I(+)N$ is a weakly $J$-ideal of
$R(+)M$.
\end{enumerate}
\end{proposition}

\begin{proof}
(1)Suppose $I(+)N$ is a weakly $J$-ideal of $R(+)M$. Then $I$ is a weakly
$J$-ideal of $R$ by \cite[Theorem 5]{Haniece3}. Now, let $r\in R$ and $m\in M$
such that $0\neq rm\in N$ and $r\notin(J(R)M:M)$. Then $(0,0)\neq
(r,0)(0,m)=(0,rm)\in I(+)N$ and clearly $(r,0)\notin J(R(+)M)$. Therefore,
$(0,m)\in I(+)N$ and so $m\in N$ as required.

(2) Let $(r_{1},m_{1}),(r_{2},m_{2})\in R(+)M$ such that $(0,0)\neq
(r_{1},m_{1})(r_{2},m_{2})\in$ $I(+)N$ and $(r_{1},m_{1})\notin J(R(+)M)$.
Then $r_{1}r_{2}\in I$ and so $r_{2}\in I$ as $I$ is a $J$-ideal and
$r_{1}\notin J(R)$. Moreover, $r_{1}m_{2}+r_{2}m_{1}\in N$ and $r_{2}m_{1}\in
IM\subseteq N$ imply that $r_{1}m_{2}\in N$. Suppose $r_{1}m_{2}=0$ but
$m_{2}\notin N$. Then by assumption $r_{2}\in(I:\left\langle r_{1}%
\right\rangle )=0$ and so $(r_{1},m_{1})(r_{2},m_{2})=(0,0)$, a contradiction.
If $r_{1}m_{2}\neq0$, then $m_{2}\in N$ since clearly $r_{1}\notin(J(R)M:M)$
and $N$ is a weakly $J$-submodule of $M$. Thus, $(r_{2},m_{2})\in$ $I(+)N$ as needed.
\end{proof}


\begin{thebibliography}{99}                                                                                               %


\bibitem {Majed}M. M. Ali, Residual submodules of multiplication modules,
Beitr Algebra Geom, 46 (2) (2005), 405--422.

\bibitem {Anderson4}D. D. Anderson, M. Winders, Idealization of a Module,
Journal of Commutative Algebra, 1 (1) (2009), 3-56.

\bibitem {At}S. E. Atani, F. Farzalipour, On weakly prime submodules, Tamkang
Journal of Mathematics, 38 (3) (2007), 247-252.

\bibitem {At2}S. E. Atani, F. Farzalipour, On weakly primary ideals, Georgian
Mathematical Journal, 12 (3)\ (2005), 423-429.

\bibitem {Bernard}A. Barnard, Multiplication modules. J. Algebra, 71 (1981), 174--178.

\bibitem {Bast}Z. A. El-Bast and P. F. Smith, Multiplication modules, Comm. in
Algebra, 16 (1988), 755-779.

\bibitem {Hani}H. A. Khashan, A. B. Bani-Ata, $J$-ideals of commutative rings,
International Electronic Journal of Algebra, 29 (2021), 148-164.

\bibitem {Haniece3}H. A. Khashan, E. Yetkin Celikel, , Weakly $J$-ideals of
commutative rings (submitted).

\bibitem {Lee}T. K. Lee, .Y. Zhou, Reduced modules. Rings, modules, algebras
and abelian groups, 236 (2004). 365-377.

\bibitem {Sharp}R. Y. Sharp, Steps in commutative algebra, Second edition,
Cambridge, UK: Cambridge University Press. (2000).

\bibitem {Smith}P. Smith, Some remarks on multiplication modules, Arch. Math.,
50 (1988), 223-235.

\bibitem {Tekir}U. Tekir, S. Koc, K. H. Oral, $n$-ideals of Commutative Rings,
Filomat, 31 (10) (2017), 2933-2941.

\bibitem {Haniece2}E. Yetkin Celikel, H. A. Khashan, Quasi $J$-submodules (submitted).
\end{thebibliography}
\end{document}